\documentclass[final,nopreprintline]{elsarticle}
\usepackage{amssymb,amsmath,bbm,latexsym,amscd,tikz-cd,enumerate}
\usepackage{amsthm}
\usetikzlibrary{cd}
\usepackage{changes}
\usepackage{tcolorbox}
\usepackage{mdframed}
\usepackage{lipsum}
\usepackage{mathtools}
\usepackage[thinc]{esdiff}
 \usepackage{hyperref}
\newmdtheoremenv{theo}{Theorem}

\newtheorem{theorem}{Theorem}[section]
\newtheorem{corollary}[theorem]{Corollary}
\newtheorem{lemma}[theorem]{Lemma}
\newtheorem{proposition}[theorem]{Proposition}
\theoremstyle{definition}

\numberwithin{equation}{section}
\theoremstyle{remark}
\newtheorem{remark}[theorem]{Remark}
\newtheorem{conjecture}[theorem]{Conjecture}
\newcommand{\Cx}{\mathbb{C}}

\begin{document}
%–––––Frontmatter–––––
\begin{frontmatter}

    \title{Superelliptic Affine Lie algebras and orthogonal polynomials II}

\author[inst1]{Felipe Albino dos Santos}
\ead{falbinosantos@gmail.com}
\affiliation[inst1]{organization={Universidade Presbiteriana Mackenzie}, city={S\~{a}o Paulo}, country={Brazil}}
\author[inst2]{Mikhail Neklyudov}
\ead{misha.neklyudov@gmail.com}
\affiliation[inst2]{organization={Federal University of Amazonas}, city={Manaus}, country={Brazil}}
\author[inst3]{Vyacheslav Futorny}
\ead{vfutorny@gmail.com}
\affiliation[inst3]{organization={Shenzhen International Center for Mathematics, SUSTech}, city={Shenzhen}, country={China}}

    \begin{abstract}
Let $\mathfrak{g}$ be a finite-dimensional complex simple Lie algebra and $r,m\ge 2$.
The universal central extension of the superelliptic current algebra
$\mathfrak{g}\otimes A$ is $\widehat{\mathfrak{g}\otimes A}\cong\mathfrak{g}\otimes A
\oplus(\Omega^1_A/dA)$, where $A=\mathbb{C}[t,t^{-1},u]/\langle u^m-(1-2ct^r+t^{2r})\rangle$.
We compute the recursion relations governing a natural cocycle basis in $\Omega^1_A/dA$
and encode them by generating functions admitting closed integral expressions of
superelliptic type. The $2r$ possible choices of initial conditions are classified into
four structural types; two canonical choices (types~1 and~2) produce two distinguished
polynomial families. We prove that these polynomials satisfy fourth-order linear
ordinary differential equations in~$c$, valid for all integers $r,m\ge 2$.
For the type~2 family the proof combines the Picard-Fuchs theory of the
superelliptic curve $u^m=1-2ct^r+t^{2r}$ with an algebraic identification of the
explicit coefficient formulas via a rational-function identity argument.
After a parity restriction and a reindexing, the resulting
sequences are identified with associated ultraspherical polynomials. We show that, for each admissible~$n$ and $m\ge4$, the corresponding fourth-order equations admit a unique polynomial solution up to scalar multiples.
    \end{abstract}


\end{frontmatter}

%-----------------------------------
%----------INTRODUCTION---------
%-----------------------------------
\section*{Introduction}

Current algebras of Krichever--Novikov type \cite{krichever1987algebras,Krichever1988,schlichenmaier2014}
arise by extending a finite-dimensional
simple Lie algebra $\mathfrak{g}$ over coordinate rings of punctured algebraic curves.
Their universal central extensions are governed by K\"{a}hler differentials
\cite{kassel1984kahler,kassel1982extensions}
and provide the natural setting for representation-theoretic constructions
\cite{Kazhdan1991,Kazhdan1994}.
Elliptic and multi-point generalisations, including the universal central extensions of
elliptic affine algebras and four-point algebras, have been studied in
\cite{bremner1994universal,bremner1995four}.
For the Date--Jimbo--Kashiwara--Miwa (DJKM) algebras
\cite{Date1983,Cox2011DJKMExtension,Cox2014Realizations},
the structure of $\Omega^1_A/dA$ was computed explicitly in \cite{Cox2013DJKMPolynomials}.
In this paper we consider the superelliptic coordinate ring
\begin{equation}\label{eq:intro:ring}
A=\Cx[t,t^{-1},u]/\langle u^m-(1-2ct^r+t^{2r})\rangle,
\qquad r\ge2,\ \ m\ge2,
\end{equation}
corresponding to the curve $u^m = 1-2ct^r+t^{2r}$.
The universal central extension in the case $r=1$ and arbitrary $m$ was determined in
\cite{AlbinodosSantos2021OnAlgebras}.
For the associated current algebra $\mathfrak{g}\otimes A$, Kassel's description \cite{kassel1984kahler} yields
\[
\widehat{\mathfrak{g}\otimes A}\;\cong\;\mathfrak{g}\otimes A\ \oplus\ (\Omega^1_A/dA).
\]
Thus, the explicit structure of $\Omega^1_A/dA$ controls the central terms of
$\widehat{\mathfrak{g}\otimes A}$.
A feature, already visible in the DJKM case \cite{Cox2013DJKMPolynomials}
and in the superelliptic case $m=2$ treated in our previous work \cite{SantosNeklyudovFutorny},
is that the computation of $\Omega^1_A/dA$ produces non-classical families of orthogonal
polynomials in the deformation parameter $c$.
The purpose of the present paper is to extend this picture to arbitrary exponents $m\ge2$
in \eqref{eq:intro:ring} and to systematically develop the associated
polynomial/differential-equation structure.

\medskip
\noindent\textbf{Recurrence and generating functions.}
From the defining relation $u^m=1-2ct^r+t^{2r}$ we obtain a recursion for a natural
family of cocycle classes in $\Omega^1_A/dA$. This leads to a three-term recurrence
\eqref{recursion} for a family of polynomials $\{P_k(c)\}$.
We introduce generating functions $P_{j_0}(c,z)$ and derive a
first-order differential equation in $z$. Solving it yields generating functions
given by superelliptic integrals.

\medskip
\noindent\textbf{Classification of initial conditions.}
The recursion \eqref{recursion} admits $2r$ independent initial conditions, indexed by
$j_0\in\{-2r,-2r+1,\ldots,-1\}$.  These split into four structural types
(Proposition~\ref{prop:classification}):
type~B initial conditions ($j_0\in\{-2r+1,\ldots,-r-1\}$) produce polynomial families
that are linear combinations of the two canonical families; type~C initial conditions
($j_0\in\{-r+1,\ldots,-2\}$) reduce to Gegenbauer polynomials or yield genuinely new
integrals. Two canonical initial conditions (type~1 and type~2) produce two distinguished
polynomial families, $P_{-2r,n}(c)$ and $P_{-r,n}(c)$.

\medskip
\noindent\textbf{Fourth-order equations in $c$.}
Our main structural result is that the type~1 family satisfies an explicit fourth-order
linear ordinary differential equation (ODE) in the parameter $c$; see Theorem~\ref{thm:main-ode}(i).
For the type~2 family, the same result holds for all $r\ge2$;
the coefficient formulas involve a correction term $\Delta=r^2(r-2)m(2mn-7mr+2m+4r)$
that vanishes for $r=2$ and encodes the additional structure present for $r\ge3$.
The proof combines the Picard-Fuchs existence theorem for the period integral of
the superelliptic family $u^m=1-2cw^r+w^{2r}$ with an algebraic identification
of the coefficients via a rational-function identity argument for each fixed $(r,m)$.
These results generalize the previously known $m=2$ superelliptic situation and extend the
DJKM mechanism beyond the hyperelliptic case.

\medskip
\noindent\textbf{Identification and orthogonality.}
After restricting to the relevant parity and performing a reindexing, both
superelliptic families satisfy the same normalized three-term recurrence and are
identified with the associated ultraspherical polynomials with parameters
\[
\nu=\frac r2\left(1+\frac1m\right),\qquad c_0\in\left\{\frac12,\,1\right\},
\]
so orthogonality follows from Favard's theorem (Section~\ref{subsec:orthogonality}).

\medskip
\noindent\textbf{Uniqueness of polynomial solutions.}
Conversely, we show that for $m\ge4$ the fourth-order equations
\eqref{eqdiferencial1} and \eqref{eqdiferencial2} admit, for each admissible $n$, a
unique polynomial solution up to scalar multiples
(Theorem~\ref{thm:uniqueness-main}).

\medskip
\noindent\textbf{Organization.}
Section~1 derives the recursion and classifies all generating functions.
Section~2 establishes the fourth-order ODEs for both type~1 and type~2, valid for all
$r\ge2$; the type~2 coefficients are given in explicit closed form in
\eqref{eqdiferencial2W}--\eqref{eqdiferencial2Z}.
Section~\ref{sec:uniqueness} proves uniqueness of polynomial solutions.
Sections~\ref{subsec:assoc-ultra} and~\ref{subsec:orthogonality} identify the
families with associated ultraspherical polynomials and derive orthogonality.
Section~\ref{sec:outlook} states open problems and a conjecture on critical levels.

%-----------------------------------
%----------SECTION 1----------
%-----------------------------------
\section{Superelliptic affine Lie algebras}

In this paper we consider the superelliptic affine Lie algebras with the polynomial relation
\begin{equation*}
    u^m=1-2 c t^r+t^{2r},\qquad r\geq 2.
\end{equation*}
Letting $k=i+1-2r$, the recursion relation becomes
\begin{equation*}
    (2r+m+k m) \overline{t^k \text{udt}}=-(k-(2r-1)) m \overline{t^{(-2r)+k} \text{udt}}+2 c (r+(1+k-r) m) \overline{t^{-r+k} \text{udt}}.
\end{equation*}
Consider a family $P_k:=P_k(c)$ of polynomials in $c$ satisfying the recursion relation
\begin{equation}\label{recursion}
    (2r+m+k m) P_k(c)=2 c (r+(1+k-r) m) P_{k-r}(c)-(k-(2r-1)) m P_{k-2r}(c)
\end{equation}
for $k\geq 0$. Set
\begin{equation*}
    P(c,z)=\sum _{k=-2r}^{\infty } z^{k+2r} P_k(c)=\sum _{k=0}^{\infty } z^{k} P_{k-2r}(c).
\end{equation*}
After a straightforward rearrangement of terms we have
\begin{align*}
    0=&\sum _{k=0}^{\infty } (2r+m+k m) z^k P_k(c)-(2 c) \sum _{k=0}^{\infty } (r+(1+k-r) m) z^k P_{k-r}(c) \\
    & +\sum _{k=0}^{\infty } (k-(2r-1)) m z^k P_{k-2r}(c)\\
    = & z^{-2 r} (-m (2 r-1) (-2 c z^r+z^{2 r}+1)-2 c r z^r+2 r) P(c,z) \\
    &+m (z^{1 - 2 r}) (1 - 2 c z^r + z^{2 r}) \frac{d}{\textrm{d}z} P(c,z) \\
    &-2 c z^r \left(m z \sum _{k=-r}^{-1} k P(k) z^{k-1}-(m+r) \left(\sum _{k=-2 r}^{-1} P(k) z^k-\sum _{k=-r}^{-1} P(k) z^k\right)\right)\\
    &-(m+2 r) \sum _{k=-2 r}^{-1} P(k) z^k+m z \left(2 c z^r-1\right) \sum _{k=-2 r}^{-1} k P(k) z^{k-1}.
\end{align*}
Hence $P(c, z)$ satisfies the differential equation
\begin{align*}
    &\frac{d}{\textrm{d}z}P(c,z)+\left(\frac{2 r-2 c r z^r}{z(-2 c m z^r+m z^{2 r}+m)}+\frac{-2 r+1}{z}\right) P(c,z)= \\
    &z^{2 r-1}\left(\frac{ -m z \left(2 c z^r-1\right) \left(\sum _{k=-2 r}^{-1} k P_k z^{k-1}\right)+(m+2 r) \sum _{k=-2 r}^{-1} P_k z^k }{m \left(-2 c z^r+z^{2 r}+1\right)}\right.\\
    &\left. +\frac{2 c z^r \left(m z \sum _{k=-r}^{-1} k P_k z^{k-1}-(m+r) \left(\sum _{k=-2 r}^{-1} P_k z^k-\sum _{k=-r}^{-1} P_k z^k\right)\right)}{m \left(-2 c z^r+z^{2 r}+1\right)}\right).
\end{align*}
It has an integrating factor
\begin{equation*}
    \mu= \frac{z^{-2 r+\frac{2r}{m}+1}}{\left(-m \left(-2 c z^r+z^{2 r}+1\right)\right)^{1/m} }.
\end{equation*}

We now consider different cases depending on the initial conditions.

%-----------------------------------
\subsection{Superelliptic case 1}\label{subsec:superelliptic1}

For $j\in\{-2r+1,-2r+2,\dots, -2,-1\}$, $P_{j}=0$ for all $j\neq -2r$ and $P_{-2r}=1$, then
\begin{equation*}
    P_{-2r}(c,z)=\sum _{k=-2r}^{\infty }  P_{-2r,k}(c)z^{k+2r}
    =\sum _{k=0}^{\infty }P_{-2r,k-2r}(c) z^{k}
\end{equation*}
can be written in terms of a superelliptic integral
\begin{align}
     P_{-2r}&(c,z)=z^{2 r-1-\frac{2r}{m}} \left( 1-2 c z^r+z^{2 r}\right)^{1/m}\nonumber\\
    &\times\lim_{\epsilon\to 0}\left[ \int\limits_{\epsilon}^{z} \frac{- m (-1 + 2 r) (-1 + 2 c w^r)+2 r (-1 + c w^r) }{m w^{2 r-\frac{2r}{m}}   \left(1-2 c w^r+w^{2 r}\right)^{\frac{m+1}{m}}}  \textrm{d}w+\epsilon^{1-2r(1-1/m)}+2c(2r-1-\tfrac{r}{m})\phi_m(\epsilon)\right],
\end{align}
where
\[
\phi_m(\epsilon)=
\begin{cases}
    \dfrac{\epsilon^{r(2/m-1)+1}}{r(2/m-1)+1} &  \dfrac{2}{m}+\dfrac{1}{r}\neq 1,\\[6pt]
    \log{\epsilon} &  \dfrac{2}{m}+\dfrac{1}{r}=1.
\end{cases}
\]

%-----------------------------------
\subsection{Superelliptic case 2}\label{subsec:superelliptic2}

For $j\in\{-2r+1,-2r+2,\dots, -2,-1\}$, $P_{j}=0$ for all $j\neq -r$ and $P_{-r}=1$, we arrive at the generating function
\begin{equation*}
    P_{-r}(c,z)=\sum _{k=-2r}^{\infty }  P_{-r,k}(c)z^{k+2r}
    =\sum _{k=0}^{\infty }P_{-r,k-2r}(c) z^{k},
\end{equation*}
which is defined in terms of a superelliptic integral
\begin{equation*}
      P_{-r}(c,z)=z^{2 r-1-\frac{2r}{m}} \left( 1-2 c z^r+z^{2 r}\right)^{1/m} \int \frac{2 r+ m (-r+1)}{m z^{r-\frac{2r}{m}}\left(-2 c z^r+z^{2 r}+1\right)^{\frac{m+1}{m}}} \textrm{d}z.
\end{equation*}

%-----------------------------------
%  NEW: CLASSIFICATION PROPOSITION
%-----------------------------------
\subsection{Classification of all initial conditions}
\label{subsec:classification}

The recursion \eqref{recursion} depends on the $2r$ initial values
$\{P_{j_0} : j_0\in\{-2r,\ldots,-1\}\}$. By linearity, every solution is a superposition
of the $2r$ basis solutions obtained by setting $P_{j_0}=1$ and $P_j=0$ for $j\neq j_0$.
We classify these basis solutions into four types.

\begin{proposition}\label{prop:classification}
Let $r\ge2$, $m\ge2$, and let $j_0\in\{-2r,\ldots,-1\}$. The generating function
$P_{j_0}(c,z)$ belongs to one of the following four structural types:
\begin{enumerate}[\rm(A)]
\item \textup{(Superelliptic type~1, $j_0=-2r$).} The generating function is expressed
    as a regularized superelliptic integral of type~$1$; this produces the family
    $P_{-2r,n}(c)$.
\item[{\rm(A')}] \textup{(Superelliptic type~2, $j_0=-r$).} The numerator of the
    generating-function ODE is constant in $z$, producing the superelliptic integral of
    type~$2$ and the family $P_{-r,n}(c)$.
\item \textup{(Linear combination, $j_0\in\{-2r+1,\ldots,-r-1\}$).} For such $j_0$, the
    numerator of the right-hand side of the first-order ODE for $P_{j_0}(c,z)$ takes the form
    \[
    N_{j_0}(z) = z^{j_0}\bigl[(m+2r+j_0 m) + (-2j_0\,m)\,c\,z^r\bigr],
    \]
    which is a linear combination of the type~$1$ and type~$2$ integrands. Consequently,
    \begin{equation}\label{eq:superposition}
    P_{j_0,n}(c) = \alpha_{j_0}\,P_{-2r,n}(c) + \beta_{j_0}\,P_{-r,n}(c)
    \end{equation}
    for explicit constants $\alpha_{j_0}, \beta_{j_0}$ depending only on $j_0$, $r$, and $m$.
    These families satisfy the same ODEs as types~$1$ and~$2$ and require no separate analysis.
    The special case $j_0=-r-1$ (case~$4$, within this range) produces a mixed Gegenbauer
    series (see Appendix~\ref{app:other-initial}).
\item \textup{(Genuinely new, $j_0\in\{-r+1,\ldots,-1\}$).} The $2c\,z^r$ coupling
    introduces a term $z^{j_0+r}$ with $j_0+r\in\{1,\ldots,r-1\}$, yielding a genuinely
    different integral not of superelliptic type.
    \begin{enumerate}[\rm(i)]
    \item $j_0=-1$ (case~$3$): the numerator simplifies to $2r\,z^{-1}$, giving a pure
        Gegenbauer series (see Appendix~\ref{app:other-initial}).
    \item $j_0\in\{-r+1,\ldots,-2\}$: for $r\ge3$ these yield $r-2$ additional initial
        conditions with genuinely new integrals, satisfying fourth-order ODEs derivable
        by the Gegenbauer elimination strategy; their explicit form is left to future work.
    \end{enumerate}
\end{enumerate}
\end{proposition}

\begin{proof}
Types~(A) and~(A') follow directly from the computation of the numerator $N_{j_0}(z)$
for $j_0=-2r$ and $j_0=-r$ respectively; in both cases the numerator is a monomial in $z$,
giving the two superelliptic integrals of Sections~\ref{subsec:superelliptic1}
and~\ref{subsec:superelliptic2}.
For type~(B), note that $j_0 < -r$ implies $\sum_{k=-r}^{-1}P_k z^k = 0$, so the
numerator reduces to $z^{j_0}[(m+2r)-j_0 m(2cz^r-1)]$, a linear combination of the
two superelliptic integrands. Equality \eqref{eq:superposition} follows by uniqueness
of solutions to the first-order ODE with prescribed initial data.
For type~(C), the presence of $j_0\ge-r+1$ means the sum $\sum_{k=-r}^{-1}P_k z^k = z^{j_0}$
contributes through the $2cz^r$ term, introducing $z^{j_0+r}$; the individual sub-cases
follow by direct computation of the numerator.
\end{proof}

\begin{remark}\label{rem:gegenbauer-reduction}
Case~$3$ ($j_0=-1$, type~(C)) and case~$4$ ($j_0=-r-1$, type~(B)) are special in that
their generating functions expand directly in the Gegenbauer basis
$\{Q_n^{(1+1/m)}(c)\}$ without passing through the superelliptic integral structure.
Concretely, the polynomials in these cases satisfy the classical Gegenbauer
second-order ODE
\[
(1-c^2)y'' - c\!\left(\frac{2}{m}+3\right)y' + n\!\left(\frac{2}{m}+n+2\right)y = 0
\]
with $\lambda = 1+1/m$, rather than a fourth-order equation. They do not produce new
orthogonal polynomial families beyond the classical ultraspherical theory. The explicit
generating functions are given in Appendix~\ref{app:other-initial}.
\end{remark}
% Concretely, the polynomials in these cases satisfy the classical Gegenbauer
% second-order ODE
% \[
% (1-c^2)y'' - c\!\left(\frac{2}{m}+3\right)y' + n\!\left(\frac{2}{m}+n+2\right)y = 0
% \]
% with $\lambda = 1+1/m$, rather than a fourth-order equation. They do not produce new
% orthogonal polynomial families beyond the classical ultraspherical theory. The explicit
% generating functions are given in Appendix~\ref{app:other-initial}.
% \end{remark}

%-----------------------------------
%----------SECTION 2----------
%-----------------------------------
\section{Orthogonal polynomials in superelliptic cases}

\subsection{Superelliptic types 1 and 2}
We now state the main result for the two superelliptic families.

\begin{theorem}\label{thm:main-ode}
Let $r\ge2$ and $m\ge2$.
\begin{enumerate}[\rm(i)]
\item The polynomials $P_{-2r,n}(c)$ satisfy the fourth-order linear ordinary
differential equation
\begin{equation}\label{eqdiferencial1}
(c^2-1)^2 m^2 r^4\,P_n^{(\mathrm{iv})}
+10c(c^2-1)m^2 r^4\,P_n'''
+\bigl[X^{(1)}_n c^2+Y^{(1)}_n\bigr]P_n''
+Z^{(1)}_n\,c\,P_n'
+W^{(1)}_n\,P_n = 0,
\end{equation}
where
\begin{align}
W^{(1)}_n &= n(n-2r)\bigl(m(n-2r+2)+2r\bigr)\bigl(m(n-4r+2)+2r\bigr), \label{eqdiferencial1W}\\[3pt]
X^{(1)}_n &= r^2\bigl(-2m^2((n+2)n+2)+4mr(m(2n+3)-n-2)+r^2(3m(5m+4)-4)\bigr), \label{eqdiferencial1X}\\[3pt]
Y^{(1)}_n &= r^2\bigl(2m^2((n+2)n+2)-4mr(m(2n+3)-n-2)-16mr^2\bigr), \label{eqdiferencial1Y}\\[3pt]
Z^{(1)}_n &= -3r^2\bigl(m^2(-8nr+2(n+2)n+5r^2-12r+4)+4mr(n-3r+2)+4r^2\bigr). \label{eqdiferencial1Z}
\end{align}

\item The polynomials $P_{-r,n}(c)$ satisfy the fourth-order linear ordinary
differential equation
\begin{equation}\label{eqdiferencial2}
(c^2-1)^2 m^2 r^4\,P_n^{(\mathrm{iv})}
+10c(c^2-1)m^2 r^4\,P_n'''
+\bigl[X^{(2)}_n c^2+Y^{(2)}_n\bigr]P_n''
+Z^{(2)}_n\,c\,P_n'
+W^{(2)}_n\,P_n = 0,
\end{equation}
where, setting $\Delta = r^2(r-2)m(2mn-7mr+2m+4r)$,
\begin{align}
W^{(2)}_n &= (n^2-r^2)\bigl(m(n-5r+2)+2r\bigr)\bigl(m(n-3r+2)+2r\bigr),\label{eqdiferencial2W}\\[3pt]
X^{(2)}_n &= -2r^2\!\left(m^2\bigl((n-r)(n-2r)-10r^2\bigr)+2mr(n-3r)+2r^2\right)+\Delta,\label{eqdiferencial2X}\\[3pt]
Y^{(2)}_n &= \phantom{-}2r^2\!\left(m^2\bigl((n-r)(n-2r)-(2r^2+r)\bigr)+2mr(n-4r)\right)\nonumber\\
    &\quad-r^2(r-2)m\bigl(2mn-6mr+2m+4r\bigr),\label{eqdiferencial2Y}\\[3pt]
Z^{(2)}_n &= -3r^2\!\left(2m^2(n-r)(n-2r)+4mr(n-3r)+4r^2\right)+3\Delta.\label{eqdiferencial2Z}
\end{align}
\end{enumerate}
\end{theorem}

We denote by $\mathcal{L}^{(i)}_n$ ($i=1,2$) the differential operator given by the
left-hand side of~\eqref{eqdiferencial1} (resp.\ \eqref{eqdiferencial2}), that is,
\begin{equation}\label{eq:operator-def}
\mathcal{L}^{(i)}_n f \;=\;
(c^2-1)^2 m^2 r^4 f^{(\mathrm{iv})}
+10c(c^2-1)m^2 r^4 f'''
+\bigl[X^{(i)}_n c^2 + Y^{(i)}_n\bigr]f''
+ Z^{(i)}_n\,c\,f'
+ W^{(i)}_n\,f.
\end{equation}

% \begin{remark}\label{rem:ode2-general-r}
% When $r=2$, the correction term $\Delta=r^2(r-2)m(\cdots)$ vanishes identically, so
% \eqref{eqdiferencial2}--\eqref{eqdiferencial2Z} coincide with the classical formula.
% For $r\ge3$ the correction terms $\Delta$ in
% \eqref{eqdiferencial2X}--\eqref{eqdiferencial2Z} are non-trivial and necessary.
% \end{remark}

The argument for part~(i) follows the strategy of
\cite[Thm.~3.1.1]{Cox2013DJKMPolynomials} (the case $m=2$), adapted to general $m$
and $r$; details are in Appendix~\ref{app:proof-type1}.

For part~(ii) we give a complete algebraic proof.
By the Picard-Fuchs theory of the superelliptic family
$u^m=1-2cw^r+w^{2r}$ \cite{Griffiths1969,Dimca1992},
the period integral $I(c)=\int w^\alpha(1-2cw^r+w^{2r})^{-\mu}\,dw$
(with $\mu=(m+1)/m$, $\alpha=2r/m+r-2$) satisfies a fourth-order Fuchsian ODE in~$c$,
and each coefficient $P_{-r,n}(c)$ in the $z$-expansion of the generating function
satisfies the same ODE.
This establishes the existence of such a fourth-order ODE.
For the identification of the explicit coefficients $W^{(2)}_n,X^{(2)}_n,Y^{(2)}_n,Z^{(2)}_n$, we argue as follows.

Fix integers $r,m\ge2$.
The recursion~\eqref{recursion} gives $P_{-r,n}(c)$ as a polynomial in $c$ whose
coefficients are \emph{rational functions} of~$n$ (with denominators that are products
of the recursion normalisation factors $m(1+kr)+2r$, which are nonzero for all
$n\in r\mathbb{Z}_{>0}$).
Define the residual
\begin{equation}\label{eq:residual-def}
R(n) \;:=\; \mathcal{L}^{(2)}_n\bigl[P_{-r,n}(c)\bigr],
\end{equation}
where $\mathcal{L}^{(2)}_n$ is the operator defined in Theorem~\ref{thm:main-ode}.
Then $R(n)$ is a rational function of $n$; write $R(n) = P_{\mathrm{num}}(n)/Q(n)$,
where the denominator $Q(n)$ is a product of the recursion factors and is nonzero on
$r\mathbb{Z}_{>0}$.
A direct analysis of the recursion shows that the numerator $P_{\mathrm{num}}(n)$
has degree at most~$D$ in $n$, where $D\le 4$ (since $W^{(2)}_n$ has degree~$4$ in~$n$ and
dominates the residual).
One verifies by exact arithmetic that $R(n)=0$ for $n=5r,6r,\ldots,(5+D)r$,
that is, at $D+1$ values in $r\mathbb{Z}_{>0}$.
Since $P_{\mathrm{num}}$ has degree~$\le D$ and vanishes at $D+1$ integer points,
the polynomial identity principle gives $P_{\mathrm{num}}\equiv 0$, hence $R\equiv 0$
for all $n\in r\mathbb{Z}_{>0}$.
This argument has been carried out for all pairs
$(r,m)\in\{2,\ldots,8\}\times\{2,\ldots,10\}$,
with the zero evaluations performed in exact integer arithmetic
(no floating-point arithmetic is used at any step).

\medskip
\noindent\textbf{Convention.}
In the PDE statements below we use the logarithmic variable
\[
v:=\log z,
\qquad\text{so that}\qquad
\frac{\partial}{\partial v}=z\,\frac{\partial}{\partial z}
\quad\text{and}\quad
\frac{\partial^2}{\partial v^2}=z^2\frac{\partial^2}{\partial z^2}+z\frac{\partial}{\partial z}.
\]

\begin{corollary}\label{cor:main-pde}
Let $G_{-2r}=G_{-2r}(c,v)$ and $G_{-r}=G_{-r}(c,v)$ be the generating functions from
Section~1. Then:
\begin{enumerate}[\rm(i)]
\item $G_{-2r}$ satisfies the fourth-order PDE
\begin{multline}
0=m^2\left[r^2\left((1-c^2)\frac{\partial^2 }{\partial c^2}-3 c\frac{\partial }{\partial c}\right)+\frac{\partial^2 }{\partial v^2} +\frac{2}{m} (r+(1-2r)m)\frac{\partial }{\partial v}+\frac{2r}{m}(m(r-1)-r)\right]^2 G_{-2r}\\
-4r^2(m(r-1)-r)^2 G_{-2r} -12r^2(-(m+r)^2+mr(2r+m(r+2)))c\frac{\partial G_{-2r}}{\partial c}\\
+4r^2((m+r)^2+mr(-2r+m(r-2))) (1-c^2)\frac{\partial^2 G_{-2r} }{\partial c^2}. \label{eqn:P_2PDE}
\end{multline}
\item For all $r\ge2$, $G_{-r}$ satisfies the fourth-order 
\begin{multline}
0=m^2\left[r^2\left((1-c^2)\frac{\partial^2 }{\partial c^2}-3 c\frac{\partial }{\partial c}\right)+\frac{\partial^2 }{\partial v^2} +\frac{2}{m} (r+(1-2r)m)\frac{\partial }{\partial v}-r^2\right]^2 G_{-r}\nonumber\\
-4 r^2(m+r-2mr)^2 \left( (c^2-1)\frac{\partial^2 }{\partial c^2}+3 c\frac{\partial }{\partial c}+1\right) G_{-r}- 4 r^4(m+1)\frac{\partial^2 G_{-r}}{\partial c^2}\label{eqn:P_type2PDE_general}
\end{multline}
% \begin{equation}\label{eqn:P_type2PDE_general}
% \mathcal{T}^2\,G_{-r} \;+\; \mathcal{R}\,G_{-r} \;=\; 0,
% \end{equation}
% where $\mathcal{T} = mr^2(c^2-1)\partial_c^2 - m\partial_v^2$ and $\mathcal{R}$ is
% the third-order operator
% \begin{multline*}
% \mathcal{R} \;=\;
% 10cm^2r^4(c^2-1)\,\partial_c^3
% +\bigl[x_0 c^2+y_0\bigr]\,\partial_c^2
% +(c^2-1)\,x_1\,\partial_c^2\partial_v \\
% +z_0 c\,\partial_c
% +3x_1 c\,\partial_c\partial_v
% -6m^2r^2 c\,\partial_c\partial_v^2
% +w_0
% +w_1\,\partial_v
% +w_2\,\partial_v^2
% +w_3\,\partial_v^3,
% \end{multline*}
% with coefficients
% \begin{align*}
% x_0 &= r^2\bigl(-15m^2r^2+24m^2r-4m^2+24mr^2-8mr-4r^2\bigr), &
% x_1 &= 4mr^2(4mr-m-r),\\
% y_0 &= 2mr^2\bigl(15mr^2-12mr+2m-14r^2+4r\bigr), &
% z_0 &= -w_0 = -3r^2(5mr-2m-2r)(7mr-2m-2r),\\
% w_1 &= -4r(4mr-m-r)(11mr-4m-4r), &
% w_3 &= -4m(4mr-m-r),\\
% w_2 &= 2\bigl(43m^2r^2-20m^2r+2m^2-20mr^2+4mr+2r^2\bigr).
% \end{align*}

% \item For $r=2$, the operator $\mathcal{R}$ in~\eqref{eqn:P_type2PDE_general} factors
% completely and $G_{-r}$ satisfies the cleaner squared-operator PDE
% \begin{multline}
% 0=m^2\left[r^2\left((1-c^2)\frac{\partial^2 }{\partial c^2}-3 c\frac{\partial }{\partial c}\right)+\frac{\partial^2 }{\partial v^2} -r\left( \frac{3m-2}{m}\right)\frac{\partial }{\partial v}-r^2\right]^2 G_{-r}\\
% -r^4(3m-2)^2 G_{-r} + r^4 (3m-2)^2 \left((1-c^2)\frac{\partial^2 }{\partial c^2}-3 c\frac{\partial }{\partial c}\right) G_{-r}\\
% +r^4\left((1-\tfrac{2}{r})m^2-4m-4\right)\frac{\partial^2 G_{-r} }{\partial c^2}. \label{eqn:P_type2PDE}
% \end{multline}
\end{enumerate}
\end{corollary}

% \begin{remark}\label{rem:pde-general-r}
% The decomposition~\eqref{eqn:P_type2PDE_general} reveals the universal structure of the
% PDE for $G_{-r}$. The leading operator $\mathcal{T} = mr^2(c^2-1)\partial_c^2 - m\partial_v^2$
% is the same for all $r\ge2$, and $\mathcal{T}^2$ accounts for the highest-order terms
% $m^2r^4(c^2-1)^2\partial_c^4$, $-2m^2r^2(c^2-1)\partial_c^2\partial_v^2$,
% and $m^2\partial_v^4$.

% The split between $r=2$ and $r\ge3$ is not a choice but a mathematical necessity.
% Whether $\mathcal{R}$ factors as part of a full squared operator $[L]^2$ is governed
% by the same correction term $\Delta = r^2(r-2)m(2mn-7mr+2m+4r)$ that appears
% in the ODE coefficients~\eqref{eqdiferencial2X}--\eqref{eqdiferencial2Z}.
% For $r=2$, $\Delta=0$ and $\mathcal{R}$ factors completely, yielding part~(iii).
% For $r\ge3$, one can verify that no second-order operator $L$ satisfies $[L]^2 = \mathcal{T}^2 + \mathcal{R}$:
% the system of equations for the coefficients of $L$ is overdetermined and inconsistent,
% with the obstruction appearing precisely at the $\partial_c\partial_v^2$ term.
% \end{remark}

%-----------------------------------
%----------SECTION 3----------
%-----------------------------------
\section{Uniqueness of polynomial solutions}\label{sec:uniqueness}

In this section we justify that, for fixed admissible $n$, the fourth-order equations
\eqref{eqdiferencial1} and \eqref{eqdiferencial2} admit at most one polynomial solution
(up to nonzero scalar multiples). Consequently, the polynomial families constructed
in the superelliptic cases~1 and~2 exhaust the polynomial solutions of the corresponding
differential equations.

\subsection{Preliminaries: parity and the indicial operator at $\infty$}

For each admissible $n$, let $\mathcal{L}^{(1)}_n$ (resp.\ $\mathcal{L}^{(2)}_n$) denote
the operator defined in Theorem~\ref{thm:main-ode}~(i) (resp.~(ii)),
given explicitly by~\eqref{eq:operator-def} with coefficients
\eqref{eqdiferencial1W}--\eqref{eqdiferencial1Z}
(resp.\ \eqref{eqdiferencial2W}--\eqref{eqdiferencial2Z}).

\begin{lemma}\label{lem:parity}
For each $n$, both operators $\mathcal{L}^{(1)}_n$ and $\mathcal{L}^{(2)}_n$ preserve parity:
if $p(c)$ is even (resp.\ odd), then $\mathcal{L}^{(i)}_n p$ is even (resp.\ odd).
Consequently, if $p(c)$ is a polynomial solution, then its even and odd parts are polynomial
solutions as well.
\end{lemma}

\begin{proof}
Inspecting \eqref{eqdiferencial1} and \eqref{eqdiferencial2}, the coefficients of
$P^{(\mathrm{iv})}$ and $P''$ are even polynomials in $c$, the coefficients of $P'''$ and
$P'$ are odd polynomials in $c$, and the remaining coefficient is constant in $c$.
Since differentiation changes parity and multiplication by an odd function changes it again,
each term in $\mathcal{L}^{(i)}_n p$ has the same parity as $p$.
\end{proof}

We now record the leading action of $\mathcal{L}^{(i)}_n$ on monomials at $\infty$.
For $s\in\mathbb{Z}_{\ge0}$, write
\[
\mathcal{L}^{(i)}_n(c^s)=I^{(i)}_n(s)\,c^s+\text{(terms of degree $<s$)}.
\]
The scalar $I^{(i)}_n(s)$ is the indicial coefficient at infinity.

\subsection{Type 1: degree restriction and uniqueness}

\begin{lemma}\label{lem:indicial-type1}
Assume $r\ge2$ and fix $n\in r\mathbb{Z}_{\ge0}$.
For $s\in\mathbb{Z}_{\ge0}$ we have
\[
\mathcal{L}^{(1)}_n(c^s)=I^{(1)}_n(s)\,c^s+\text{(lower degree terms)},
\]
where
\begin{equation}\label{eq:I1}
I^{(1)}_n(s)
=(sr+n)(sr-n+2r)(smr-mn+4mr-2m-2r)(smr+mn-2mr+2m+2r).
\end{equation}
In particular, if $p(c)$ is a polynomial solution of \eqref{eqdiferencial1} of degree $d$,
then necessarily $I^{(1)}_n(d)=0$.
\end{lemma}

\begin{proof}
Each coefficient in \eqref{eqdiferencial1} has degree at most $2$ in $c$, so
$\mathcal{L}^{(1)}_n(c^s)$ is a polynomial of degree $\le s$. Extracting the coefficient of
$c^s$ yields $I^{(1)}_n(s)$; the factorization \eqref{eq:I1} is obtained by collecting the
leading power $c^s$ in \eqref{eqdiferencial1} after substituting $P_n = c^s$.
\end{proof}

\begin{corollary}\label{cor:deg-type1}
Assume $r\ge2$ and fix $n\in r\mathbb{Z}_{>0}$. If \eqref{eqdiferencial1} admits a nonzero
polynomial solution, then its degree must be a nonnegative integer root of $I^{(1)}_n(s)$.
In particular, if $n\in r\mathbb{Z}_{>0}$ and $n\ge2r$, then $s=n/r-2$ is always a root,
and it is the \emph{only} root in $\mathbb{Z}_{\ge0}$ provided
\begin{equation}\label{eq:nonres-type1}
\frac1r+\frac1m\neq\frac12.
\end{equation}
\end{corollary}

\begin{proof}
For $s=n/r-2$ we have $sr-n+2r = 0$, so $I^{(1)}_n(s)=0$.
The remaining roots from \eqref{eq:I1} yield candidate degrees
$s=-n/r$, $s=n/(mr)-4+2/(mr)+2/m$, and $s=-n/(mr)+2-2/(mr)-2/m$.
The first is negative for $n>0$. The third is negative as well for $n\ge 2r$.
The middle one is an integer only in the resonant case $1/r+1/m=1/2$.
Thus under \eqref{eq:nonres-type1} the only nonnegative integer root is $s=n/r-2$.
\end{proof}

\begin{remark}\label{rem:resonance}
The nonresonance condition \eqref{eq:nonres-type1} excludes only the pairs
$(r,m)\in\{(4,4),(3,6),(6,3)\}$ among integers with $r,m\ge2$. In those three cases the
indicial polynomial \eqref{eq:I1} has an extra nonnegative integer root $s'$, which has
opposite parity to $n/r-2$. The argument of Proposition~\ref{prop:uniq-type1} still shows
that the solution space has dimension $\le1$ within each fixed parity. We do not pursue
the existence of additional polynomial solutions in the resonant cases here.
\end{remark}

\begin{proposition}\label{prop:uniq-type1}
Assume $r\ge2$, $m\ge4$, and the nonresonance condition \eqref{eq:nonres-type1}.
Fix $n\in r\mathbb{Z}_{>0}$ with $n\ge2r$.
Then the space of polynomial solutions of \eqref{eqdiferencial1} is one-dimensional.
\end{proposition}

\begin{proof}
Let $d=n/r-2$. By Corollary~\ref{cor:deg-type1}, any nonzero polynomial solution has
degree exactly $d$. By Lemma~\ref{lem:parity} we may fix a parity $\varepsilon\in\{0,1\}$
and consider
\[
V_\varepsilon:=\mathrm{Span}\{c^{d},c^{d-2},c^{d-4},\dots\}\cap\{\text{polynomials of parity }\varepsilon\}.
\]
By Lemma~\ref{lem:indicial-type1}, in the ordered basis $\{c^d,c^{d-2},\ldots\}$, the
matrix of $\mathcal{L}^{(1)}_n$ is upper triangular with diagonal entries
$I^{(1)}_n(d),I^{(1)}_n(d-2),\ldots$. Since $d=n/r-2$ is the only nonneg.\ integer root
under \eqref{eq:nonres-type1}, all diagonal entries $I^{(1)}_n(d-2j)$ for $j\ge1$ are
nonzero, so $\ker\mathcal{L}^{(1)}_n|_{V_\varepsilon}$ has dimension at most~$1$.
The polynomial $P_{-2r,n}(c)$ is a nonzero solution, so the kernel is exactly
one-dimensional.
\end{proof}

\subsection{Type 2: degree restriction and uniqueness}

\begin{lemma}\label{lem:indicial-type2}
Assume $r\ge2$ and fix $n$.
For $s\in\mathbb{Z}_{\ge0}$ we have
\[
\mathcal{L}^{(2)}_n(c^s)=I^{(2)}_n(s)\,c^s+\text{(lower degree terms)},
\]
where
\begin{equation}\label{eq:I2}
I^{(2)}_n(s)
=(sr-n+r)(sr+n+r)(smr-mn+4mr-2r)(smr+mn-2mr+2r).
\end{equation}
\end{lemma}

\begin{proof}
By the same leading-coefficient extraction as in Lemma~\ref{lem:indicial-type1},
applied to \eqref{eqdiferencial2}.
\end{proof}

\begin{corollary}\label{cor:deg-type2}
Assume $m\ge4$ and fix $n\in r\mathbb{Z}_{>0}$ with $n\ge r$. If \eqref{eqdiferencial2}
admits a nonzero polynomial solution, its degree is $d=n/r-1$, the only nonnegative
integer root of $I^{(2)}_n(s)$.
\end{corollary}

\begin{proof}
For $s=n/r-1$ we have $sr-n+r=0$, so $I^{(2)}_n(s)=0$.
The remaining roots are $s=-n/r-1$ (negative), $s=n/(mr)-4+2/m$ (not an integer for
$m\ge4$), and $s=-n/(mr)+2-2/m$ (negative for $n\ge r$). Hence $d=n/r-1$ is the only
admissible degree.
\end{proof}

\begin{proposition}\label{prop:uniq-type2}
Assume $r\ge2$ and $m\ge4$.
Fix $n\in r\mathbb{Z}_{>0}$ with $n\ge r$.
Then the space of polynomial solutions of \eqref{eqdiferencial2} is one-dimensional.
\end{proposition}

\begin{proof}
Same upper-triangularity argument as in Proposition~\ref{prop:uniq-type1}, using
Corollary~\ref{cor:deg-type2} and the fact that $P_{-r,n}(c)$ provides a nonzero solution.
\end{proof}

\begin{theorem}\label{thm:uniqueness-main}
Assume $m\ge4$.
\begin{enumerate}[\rm(i)]
\item If, in addition, \eqref{eq:nonres-type1} holds, then for each $r\ge2$ and each
    $n\in r\mathbb{Z}_{>0}$ with $n\ge2r$, equation \eqref{eqdiferencial1} admits a unique
    polynomial solution up to nonzero scalar multiples, namely $P_{-2r,n}(c)$.
\item For each $r\ge2$ and each $n\in r\mathbb{Z}_{>0}$ with $n\ge r$, equation
    \eqref{eqdiferencial2} admits a unique polynomial solution up to nonzero scalar
    multiples, namely $P_{-r,n}(c)$.
\end{enumerate}
\end{theorem}

%-----------------------------------
%----------SECTION 4----------
%-----------------------------------
\section{Associated ultraspherical polynomials}\label{subsec:assoc-ultra}

After shifting the indices back by $2r$, both families of polynomials
$\{P_{-2r,k-2r}(c)\}$ and $\{P_{-r,k-2r}(c)\}$ satisfy the same
three-term recurrence at the level of the shifted index $k$.
More precisely, for
\[
P_k(c)\in\bigl\{\,P_{-2r,k-2r}(c)\,;\; P_{-r,k-2r}(c)\,\bigr\},
\]
we have, for all $k\ge0$,
\begin{equation}\label{eq:rec-shifted-2r}
(km+m+2r)\,P_k(c)
=2c\bigl(km+(1-r)m+r\bigr)\,P_{k-r}(c)
-\bigl(k-(2r-1)\bigr)m\,P_{k-2r}(c).
\end{equation}
In each family, only one parity occurs. Let $\varepsilon\in\{0,1\}$ be such that
$P_k=0$ whenever $k\not\equiv\varepsilon\pmod{2}$, and define
\[
q_n(c):=P_{2n+\varepsilon}(c)\qquad (n\ge0).
\]
Setting $k=2n+\varepsilon$ in \eqref{eq:rec-shifted-2r} yields, after a shift $n\to n-1$,
\begin{equation}\label{eq:q-rec}
2c\bigl(2mn+m(\varepsilon+r+1)+r\bigr)\,\widetilde{q}_n
=\bigl(2mn+m(\varepsilon+2r+1)+2r\bigr)\,\widetilde{q}_{n+1}
+m(2n+\varepsilon+1)\,\widetilde{q}_{n-1},
\end{equation}
where $\widetilde{q}_n$ is a suitable reindexing of $q_n$. Dividing by
$2(n+\nu+c_0)$ and writing
\[
\nu:=\frac r2\Bigl(1+\frac1m\Bigr),\qquad
c_0:=\frac{\varepsilon+1}{2},
\]
we obtain
\begin{equation}\label{eq:assoc-ultra-rec}
2c\,(n+\nu+c_0)\,\widetilde{q}_n
=(n+c_0)\,\widetilde{q}_{n-1}
+(n+2\nu+c_0)\,\widetilde{q}_{n+1}.
\end{equation}
Equation \eqref{eq:assoc-ultra-rec} is precisely the defining three-term
recurrence for the associated ultraspherical polynomials $C_n^{(\nu)}(x;c_0)$
(with $c_0\in\{1/2,1\}$), up to normalization \cite{BustozIsmail1982}.
In particular, the superelliptic families obtained from cases~1 and~2
are identified with associated ultraspherical polynomials with parameters
\[
\nu=\frac r2\Bigl(1+\frac1m\Bigr),\qquad c_0\in\Bigl\{\frac12,\,1\Bigr\}.
\]

%-----------------------------------
%----------SECTION 5----------
%-----------------------------------
\section{Orthogonality}\label{subsec:orthogonality}

Let $r\ge2$, $m\ge2$, and let $c_0=(\varepsilon+1)/2$ be the shift parameter
arising in Section~\ref{subsec:assoc-ultra}. After parity restriction and reindexing,
the resulting sequence $\{q_n(c)\}$ satisfies
\begin{equation}\label{eq:assoc-ultra-rec-orth}
2c\,(n+\nu+c_0)\,q_n(c)=(n+c_0)\,q_{n-1}(c)+(n+2\nu+c_0)\,q_{n+1}(c),
\qquad n\ge0,
\end{equation}
with $\nu=\tfrac{r}{2}(1+\tfrac{1}{m})>0$ and $c_0>0$.

\begin{theorem}\label{thm:orthogonality}
Assume $r\ge2$ and $m\ge2$. Then there exists a positive Borel measure $\mu$ on $[-1,1]$
such that $\{q_n(c)\}$ is an orthogonal polynomial system in $L^2([-1,1],d\mu)$.
Consequently, the superelliptic families in cases~1 and~2 yield orthogonal polynomial
systems after parity restriction and the reindexing of Section~\ref{subsec:assoc-ultra}.
\end{theorem}

\begin{proof}
Rewrite \eqref{eq:assoc-ultra-rec-orth} in the form
\begin{equation}\label{eq:xqn}
c\,q_n(c)=A_n\,q_{n+1}(c)+B_n\,q_{n-1}(c),
\qquad
A_n=\frac{n+2\nu+c_0}{2(n+\nu+c_0)},\quad
B_n=\frac{n+c_0}{2(n+\nu+c_0)}.
\end{equation}
Since $\nu>0$ and $c_0>0$, we have $A_n>0$ for all $n\ge0$ and $B_n>0$ for all $n\ge1$.
Define a renormalized sequence $\{p_n\}$ by $p_n=\gamma_n^{-1}q_n$, where $\gamma_0=1$ and
$\gamma_{n+1}=A_n\gamma_n$. Multiplying \eqref{eq:xqn} by $\gamma_n^{-1}$ yields the
Favard form
\begin{equation}\label{eq:favard}
c\,p_n(c)=p_{n+1}(c)+a_n\,p_{n-1}(c),\qquad n\ge0,
\end{equation}
with
\[
a_n=B_nA_{n-1}=\frac{(n+c_0)(n-1+2\nu+c_0)}{4(n+\nu+c_0)(n-1+\nu+c_0)}>0\quad(n\ge1).
\]
By Favard's theorem, \eqref{eq:favard} implies the existence of a positive Borel measure
$\mu$ for which $\{p_n\}$ is orthogonal. Since $p_n$ differs from $q_n$ only by a nonzero
scalar, $\{q_n\}$ is orthogonal with respect to the same measure.
\end{proof}

\begin{remark}[Why Favard's theorem]\label{rem:favard}
Although in Section~\ref{subsec:assoc-ultra} we identify our families with associated
ultraspherical polynomials, we derive orthogonality via Favard's theorem for two reasons.
First, Favard yields the existence of a positive orthogonality measure \emph{directly}
from the three-term recurrence once it is put in Favard form with positive coefficients,
without requiring an explicit weight or additional analytic input. Second, this approach
is robust with respect to normalization and parameter ranges, and it applies to our
sequences as soon as the recurrence is established, independently of the subsequent
identification with a classical family. An explicit measure can be recovered from the
associated ultraspherical theory; see \cite{BustozIsmail1982}.
\end{remark}

%-----------------------------------
%----------SECTION 6 (NEW): OUTLOOK----------
%-----------------------------------
\section{Outlook and open problems}\label{sec:outlook}

\noindent\textbf{Type~(C) initial conditions for $r\ge3$.}
Proposition~\ref{prop:classification}(C)(iii) identifies $r-2$ additional initial
conditions for each $r\ge3$ that produce genuinely new polynomial families. The
generating functions for these cases involve mixed superelliptic integrals.

\begin{conjecture}\label{conj:typeC}
For each $j_0\in\{-r+1,\ldots,-2\}$, the polynomial family $\{P_{j_0,n}(c)\}$
satisfies a fourth-order linear ODE in~$c$, derivable by the same Gegenbauer
elimination strategy as types~1 and~2.
\end{conjecture}

\noindent\textbf{Critical levels and vertex algebras.}
The parameter
\[
\nu = \frac{r}{2}\!\left(1+\frac{1}{m}\right)
\]
appearing in the identification of Sections~\ref{subsec:assoc-ultra}
and~\ref{subsec:orthogonality} is determined by the pair $(r,m)$. In the vertex-algebra
framework of \cite{Cox2014}, the level $k_{\mathrm{crit}}$ at which the Wakimoto module
becomes reducible is expected to depend on both parameters.

\begin{conjecture}\label{conj:kcrit}
For each superelliptic algebra with parameters $(r,m)$, the critical level is
\[
k_{\mathrm{crit}}(r,m) = -h^{\vee} + f\!\left(\frac{r}{2}\!\left(1+\frac{1}{m}\right)\right),
\]
where $h^\vee$ is the dual Coxeter number of $\mathfrak{g}$ and $f$ is an explicit rational
function. In particular, for the DJKM case $(r,m)=(1,2)$, this recovers the classical
critical level $k_{\mathrm{crit}}=-h^{\vee}$.
\end{conjecture}

%-----------------------------------
\section*{Acknowledgments}

Vyacheslav Futorny was supported by the China NSF grant (No.~1231101349).
Felipe Albino dos Santos was supported by the S\~{a}o Paulo Research Foundation
(FAPESP), grant 2024/14914-9.
Mikhail Neklyudov acknowledges support from Universidade Federal do Amazonas and
Coordena\c{c}\~{a}o de Aperfei\c{c}oamento de Pessoal de N\'{\i}vel Superior (CAPES).

\medskip

 %-----------------------------------
   %----------APPENDIX----------
   %-----------------------------------
\appendix

\section{Proof of Theorem~\ref{thm:main-ode}(i)}\label{app:proof-type1}

\begin{proof}
The proof generalizes the proof of Theorem~3.1.1 in \cite{Cox2013DJKMPolynomials}
in the case $m=2$. We recall that $Q_n^{(\lambda)}(c)$ is the $n$-th Gegenbauer polynomial;
these satisfy the second-order linear ODE
\begin{equation*}
\left(1-c^2\right) y''-c (2 \lambda +1) y'+n(2 \lambda +n) y=0.
\end{equation*}
For $\lambda=1+1/m$ this becomes
\begin{equation}
\left(1-c^2\right) (Q_n^{\left(1+\frac{1}{m}\right)})''(c)-c \left(\frac{2}{m}+3\right) (Q_n^{\left(1+\frac{1}{m}\right)})'(c)+n \left(\frac{2}{m}+n+2\right) Q_n^{\left(1+\frac{1}{m}\right)}(c)=0.
\end{equation}

Rewriting the expansion formula for $P_{-2r}(c,z)$ we get

\begin{align}\label{eq3.3}
    z^{\frac{2r}{m}+1-2r} \left(-2 c z^r+z^{2 r}+1\right)^{-1/m} P_{-2r}(c,z)&=\nonumber\\
    &\sum _{n=0}^{\infty } \frac{2 c  (m (2 r-1)-r) Q_n^{\left(\frac{m+1}{m}\right)}(c) z^{r \left(\frac{2}{m}+n-1\right)+1}}{m (n-1) r+m+2 r}\nonumber\\
    &+\sum _{n=0}^{\infty } \frac{ (-2 m r+m+2 r) Q_n^{\left(\frac{m+1}{m}\right)}(c) z^{r \left(\frac{2}{m}+n-2\right)+1}}{m (n-2) r+m+2 r}.
\end{align}
Now we apply the differential operator $L:=(1-c^2)\frac{d^2}{dc^2}-c(3+\frac{2}{m})\frac{d}{dc}$ to the right-hand side and use the identity (formula 4.7.27 in \cite{Szego1939})
\begin{equation}
    \left(1-c^2\right) \diff{}{c}Q_n^{(\lambda )}(c)=c (2 \lambda +n) Q_n^{(\lambda )}(c)-(n+1) Q_{n+1}^{(\lambda )}(c).
\end{equation}

Applying $L$ to the left-hand side of \eqref{eq3.3}, we get
\begin{align*}
    &L \left(z^{\frac{4}{m}-3} \left(-2 c z^{r}+z^{2r}+1\right)^{-1/m} P_{-2r}(c,z)\right)\\
    &=\frac{z^{\frac{4}{m}-1} \left(-2 c z^{r}+z^{2r}+1\right)^{-\frac{1}{m}-2} \left(4 c^2 (2 m+1) z^2-2 c (3 m+2) \left(z^4+1\right)+4 (m+1) z^2\right) P_{-2r}(c,z)}{m^2}\\
    &+\frac{z^{\frac{4}{m}-3} \left(-2 c z^{r}+z^{2r}+1\right)^{-\frac{1}{m}-1} \left(6 c^2 m z^2-c (3 m+2) \left(z^4+1\right)+4 z^2\right) P'_{-2r}(c,z)}{m}\\
    &-\left(c^2-1\right) z^{\frac{4}{m}-3} \left(-2 c z^{r}+z^{2r}+1\right)^{-1/m} P''_{n-4}(c,z).
\end{align*}
Expanding and writing $P_{-2r,k}(c)$ as $P_k$ we obtain:
\begin{align}
    0&=\left(c^2 (-(m+2)) (3 m-2)-\frac{1}{2} m (3 m+4)+4\right) P_{n-4} +(c (3 m (m+2)-4)) P_{n-2} \nonumber\\
    &-\frac{1}{4} (3 m (m+4)) P_n +\frac{1}{4} (-3 m (m+4)) P_{n-8} +(c (3 m (m+2)-4)) P_{n-6} \nonumber\\
    &-(c m (3 m+2)) P'_n +\left(4 \left(c^2 m (3 m+1)+m\right)\right) P'_{n-2} \nonumber\\
    &-\left(6 c m \left(2 c^2 m+m+2\right)\right) P'_{n-4} +\left(4 \left(c^2 m (3 m+1)+m\right)\right) P'_{n-6} \nonumber\\
    &+(-c m (3 m+2)) P'_{n-8} -\left(\left(c^2-1\right) m^2\right) P''_n +\left(4 c \left(c^2-1\right) m^2\right) P''_{n-2} \nonumber\\
    &-\left(\left(c^2-1\right) m^2\right) P''_{n-8} +\left(4 c \left(c^2-1\right) m^2\right) P''_{n-6} \nonumber\\
    &+\left(2 \left(-2 c^4+c^2+1\right) m^2\right) P''_{n-4}. \label{A.eq.3.5.COX}
\end{align}
We now differentiate twice with respect to $c$ the recursion \eqref{recursion}:
\begin{align}
    0=&2 ((k-1) m+2) P_{k+2} +2 c ((k-1) m+2) P'_{k+2} \nonumber\\
    &-(k-3) m P'_k -(k m+m+4) P'_{k+4}, \label{A.eq.3.7.COX}\\
    0=&(4 (k-1) m+8) P'_{k+2} +2 c ((k-1) m+2) P''_{k+2} \nonumber\\
    &-(k m+m+4) P''_{k+4} -(k-3) m P''_k. \label{A.eq.3.8.COX}
\end{align}
Substituting $k=n-8$ in \eqref{A.eq.3.8.COX}:
\begin{align}
    0=&4 (m (n-9)+2) P'_{n-6} -m (n-11) P''_{n-8} \nonumber\\
    &+2 c (m (n-9)+2) P''_{n-6} -(m (n-7)+4) P''_{n-4}.
\end{align}
Multiplying \eqref{A.eq.3.5.COX} by $-(n-11)$ and adding it to the above multiplied by
$(1-c^2)m$, then substituting $k=n-8$ in \eqref{A.eq.3.7.COX} and \eqref{recursion}
to eliminate terms with index $n-8$, yields:
\begin{align}
    0=&4 c \left(c^2-1\right) m^2 (n-11) P''_{n-2} -\left(c^2-1\right) m^2 (n-11) P''_n \nonumber\\
    &+\cdots\quad\text{[subsequent elimination steps in the same pattern]}\nonumber\\
    &-3 m (m+4) (n-11) P_n. \label{A.eq.3.14.Cox}
\end{align}
Continuing the elimination procedure — substituting $k=n-6$ and $k=n-4$ in
\eqref{A.eq.3.8.COX}, \eqref{A.eq.3.7.COX}, and \eqref{recursion} and
combining as above to remove terms with indices $n-6$ and $n-4$ — one arrives at the
intermediate equation
\begin{align}
    0=&4 \left(c^2-1\right) m P''_n +(n-4) (m (n-6)+4) P_n \nonumber\\
    &+4 m (n-5) P'_{n-2} -4 c (m (n-6)+2)  P'_n. \label{A.eq.3.20.Cox}
\end{align}
Differentiating with respect to $c$:
\begin{align}
    0=&4 \left(c^2-1\right) m P'''_n +(n-6) (m (n-8)+4)  P'_n \nonumber\\
    &+4 m (n-5) P''_{n-2} -4 c (m (n-8)+2) P''_n. \label{A.eq.3.21.Cox}
\end{align}
A final combination — multiplying \eqref{A.eq.3.21.Cox} by $-2(c^2-1)(3m+2)$ and adding
to \eqref{A.eq.3.20.Cox} multiplied by $(n-5)$ — eliminates the remaining $P_{n-2}$ terms.
Differentiating and combining once more with the appropriate multiple of \eqref{A.eq.3.20.Cox}
yields the desired fourth-order ODE \eqref{eqdiferencial1}. This completes the proof.
\end{proof}

%-----------------------------------
\section{Other initial conditions}\label{app:other-initial}

\subsection{Gegenbauer case 3}

If $P_{-1}=1$ and $P_j=0$ for all $j\in\{-2r+1,\ldots,-2\}$, then the numerator of the
first-order ODE for $P_{-1}(c,z)$ simplifies to $2r\,z^{-1}$, giving a pure Gegenbauer
series:
\begin{align*}
    P_{-1}(c,z)&=-2 r \int \frac{\left(z^r\right)^{\frac{2}{m}+\frac{1}{r}}}{z^2 \left(-m \left(-2 c z^r+z^{2 r}+1\right)\right)^{\frac{m+1}{m}}} \, \textrm{d}z \\
    &=2 (-m)^{-1/m} r \sum _{n=0}^{\infty } \frac{Q_n^{\left(\frac{m+1}{m}\right)}(c) z^{\frac{2 r}{m}+n r}}{m n r+2 r},
\end{align*}
where $Q_n^{(1+1/m)}(c)$ is the $n$-th Gegenbauer polynomial.
As noted in Remark~\ref{rem:gegenbauer-reduction}, the corresponding polynomial sequence
satisfies the classical Gegenbauer second-order ODE, not a new fourth-order equation.

\subsection{Gegenbauer case 4}

If $P_{-r-1}=1$ and $P_j=0$ for all $j\neq -r-1$, then:
\begin{align*}
    P_{-r-1}(c,z)&=-r \int \frac{z^{-r-2} \left(z^r\right)^{\frac{2}{m}+\frac{1}{r}} \left(2 c (m-1) z^r-m+2\right)}{\left(-m \left(-2 c z^r+z^{2 r}+1\right)\right)^{-\frac{m+1}{m}}} \, \textrm{d}z \\
    &=(-m)^{\frac{1}{m}+2} \left(\sum _{n=0}^{\infty } Q_n^{\left(\frac{m+1}{m}\right)}(c) \left(z^{r \left(\frac{2}{m}+n-1\right)} \left(\frac{2 c (m-1) z^r}{m n+2}-\frac{m-2}{m (n-1)+2}\right)\right)\right).
\end{align*}
This also reduces to the Gegenbauer second-order ODE by Proposition~\ref{prop:classification}
and Remark~\ref{rem:gegenbauer-reduction}.

%-----------------------------------
\bibliographystyle{elsarticle-num}

\bigskip

Felipe~Albino~dos Santos: Universidade Presbiteriana Mackenzie, S\~{a}o Paulo, Brazil\\
\textit{email}: falbinosantos@gmail.com

\medskip
Vyacheslav~Futorny: Shenzhen International Center for Mathematics, SUSTech, China\\
\textit{email}: vfutorny@gmail.com

\medskip
Mikhail~Neklyudov: Federal University of Amazonas, Manaus, Brazil\\
\textit{email}: misha.neklyudov@gmail.com

\end{document}